\documentclass[12pt]{amsart}
\usepackage{amssymb}

\evensidemargin\paperwidth
\advance\evensidemargin-\textwidth
\advance\oddsidemargin-1in

\evensidemargin\oddsidemargin

\topmargin\paperheight
\advance\topmargin-\textheight
\advance\topmargin-1in

\theoremstyle{plain}
\newtheorem{theorem}{Theorem}[section]
\theoremstyle{plain}
\newtheorem{proposition}[theorem]{Proposition}
\theoremstyle{plain}
\newtheorem{lemma}[theorem]{Lemma}
\theoremstyle{plain}

\theoremstyle{plain}

\theoremstyle{plain}
\newtheorem{definition}[theorem]{Definition}
\theoremstyle{plain}

\theoremstyle{remark}
\newtheorem{remark}[theorem]{Remark}
\theoremstyle{remark}

\theoremstyle{remark}
\newtheorem{acknowledgment}{Acknowledgment\hspace{-1mm}}

\title[Indeterminate moments]
{Indeterminacy of the moment problem for symmetric probability measures}

\author{Hayato Saigo}
\author{Hiroki Sako}

\address
{Hayato Saigo,
Nagahama Institute of Bio-Science and Technology,
Nagahama 526-0829, Japan}
\email{h\_saigoh@nagahama-i-bio.ac.jp}
\address
{Hiroki Sako, 
School of Science, Tokai University, Kanagawa, 259-1292 Japan}
\email{hiroki@tokai-u.jp}


\begin{document}
\maketitle
\begin{abstract}
In this paper, the moment problem for symmetric probability measures 
is characterized in terms of 
associated sequences called Jacobi sequences $\{\omega_n\}$. 
A notion named property (SC), 
which is proved to be a necessary and sufficient condition for the indeterminacy of the moment problem, 
naturally arises from the viewpoint of finite dimensional approximation for infinite matrices. 
We prove that the moment 
problem for q-Gaussian not only for $q>1$ but also for $q<-1$ is indeterminate.
We also prove
that hyperbolic secant distribution 
is ``the last probability measure'' which is uniquely determined by 
the moment sequence of power type, just by checking property (SC) with quite easy calculation.
\end{abstract}

\section{Introduction}
The moment sequence of a probability measure is, if it exists, one of the most important information of the measure. 
Then, does  the moment sequence determine the measure uniquely or not?---This fundamental question called 
``moment problem'' (or more precisely ``determinate moment problem'' ) has 
been attracted many mathematicians since the work by Stieltjes \cite{Stieltjes}, who proposed and solved the moment problem for $\mathbb{R}_{\geqq 0}$. It was Hamburger (\cite{Hamburger1}, \cite{Hamburger2}, and \cite{Hamburger3}) who extended the problem for $\mathbb{R}$, on which we study here.
(For the 
subject and its history, see Akhiezer \cite{Akhiezer} and Shohat--Tamarkin \cite{ShohatTamarkin}). 
If the measure is unique, the problem is said to be determinate, and otherwise, indeterminate.
In the present paper, we discuss the moment problem 
in terms of the associated sequence which we call ``Jacobi sequence''. 

The sequence of orthogonal polynomials also provides useful
data for the probability measure $\mu$ whose moments 
${\{M_n = \int_\mathbb{R} x^n d \mu\}}$ are finite.
Then the space of polynomial functions is contained in the Hilbert space $L^2(\mathbb{R},\mu )$. 
A Gram-Schmidt procedure provides orthogonal polynomials which only depend on the moment sequence.
Let $\{P_n(x) \}_{n=0, 1, \cdots}$ be the monic orthogonal polynomials of $\mu$. 
Then there exist sequences $\{\alpha_n\}_{n=0, 1, \cdots}$ 
and $\{\omega_n\}_{n= 1, 2, \cdots}$ such that 
\begin{eqnarray}\label{equation3term}
\quad x P_n(x) 
= P_{n+1}(x) + \alpha_{n+1} P_n(x) + \omega_n P_{n-1}(x), \quad P_{-1}(x) = 0.
\end{eqnarray}
These sequences $\{\omega_n\}$ and $\{\alpha_n\}$ are called ``Jacobi sequences'' associated to $\mu$. 
The Jacobi sequences have the same amount of data as the moment sequence.

In this paper, we will concentrate on the case that
$\mu$ is symmetric, i.e., $\mu(-dx)=\mu(dx)$. 
Then $\alpha_{n}$ are all zero.
We use the term 
``Jacobi sequence'' to indicate the sequence $\{\omega_n\}$.
It is also known that the real numbers $\omega_n$ are all greater than $0$ if and only if
$\mu$ has infinite support. 
Since the moment problem is trivial for measures with finite support, we always assume that $\{\omega_n\}$ are all positive.

In the present paper, a new notion ``property (SC)''
%\footnote{
%Since the announcement of this result, 
%many experts have expressed their impression that
%property (SC) has been already discovered.
%Of course, in a logical sense, there exist many conditions which are equivalent to property (SC), because the determinate moment problem is classical.
%The problem has been rephrased in terms of 
%(absolute) convergence or divergence of several kinds of sequences.
%However, the authors did not find similarity between known conditions and our property (SC).
%Our motivation for the definition of property (SC) is 
%to give a concise and self-contained proof for the indeterminacy
%for interesting measures in quantum probability theory.
%As an example, 
%our characterization implies the new theorems (Theorem \ref{introQGauss} and Theorem \ref{introTheoremHSD}).  
%} 
is defined for infinite Jacobi sequences and 
is proved to be equivalent to the indeterminacy of the moment problem. 
Section \ref{sectionfluctuation} 
is devoted to formulate the notion. In Section \ref{sectionselfadjoint}, it is shown that property (SC) leads to a construction of two distinct 
self-adjoint operators and proved to imply the indeterminacy of the moment problem.
We will investigate basic examples of Jacobi sequences with property (SC), $q$-deformed integers 
$[n]_q = 1 + q + q^2 + \cdots + q^{n -1}$ ($q>-1$), $\displaystyle{(-1)^{n-1} \frac{q^n - 1}{q - 1}}$ ($q<-1$)   
and powers of integers $n^p$ ($p>2$) in Section \ref{sectionexamples}. 
The following theorems are consequences:

\begin{theorem}\label{introQGauss}
The moment problem for $q$-Gaussians ($q\in\mathbb{R}$) are indeterminate if and only if $|q|>1$. 
\end{theorem}

In the case of $q > 1$, Ismail and Masson proved that 
the moment problem is indeterminate, identifying extremal measures \cite{IsmailMasson}.

\begin{theorem}\label{introTheoremHSD}
Let $p$ be a positive number.
The probability measure corresponding to the Jacobi sequence
$\{n^p\}$ is unique if and only if $p \leqq 2$.
\end{theorem}

\noindent
The Jacobi sequence of the the 
hyperbolic secant distribution 
$\displaystyle{\frac{dx}{e^{\pi x / 2} + e^{- \pi x / 2}}}$ on $\mathbb{R}$ is $\{n^2\}_{n = 1}^\infty$.
The theorem above means that the 
hyperbolic secant distribution 
is ``the last probability measure'' which is uniquely determined by 
the moment sequence of power type.
In Section \ref{sectiondeterminate}, we prove that property (SC)
holds 
if and only if the corresponding probability measure is not unique. 
\\

\paragraph
{\bf Notations and Conventions}
\begin{itemize}
\item 
$\mathbb{N}$ is the set of non-negative integers $\{0, 1, 2, \cdots \}$.
\item
$\ell_2(\mathbb{N})$ stands for the set of all the
square summable sequences.
\item
Identify the finite dimensional Hilbert space $\ell_2(\{0, 1, \cdots , m\})$ with the closed subspace of $\ell_2(\mathbb{N})$.
\item
We regard a vector $\eta$ in $\ell_2(\{0, 1, \cdots , m\})$ or 
in $\ell_2(\mathbb{N})$ as a column vector.
For an natural number $l$, 
$\,^l \eta$ stands for the 
$l$-th entry of the vector.
\item
For $k \in \mathbb{N}$,
$\delta_k$ stands for a unit vector given by
$\displaystyle{
\,^l \delta_k = \left\{
\begin{array}{cc}
0, & l \neq k,\\
1, & l = k.
\end{array}
\right.}$
\item
Fixing the orthonormal basis $\{\delta_k\}$, we identify
the set of linear operators 
$\mathbb{B}(\ell_2(\{0, 1, \cdots, m\}))$ 
with the set of complex matrices $\mathbb{M}(m + 1)$.
\item
For an operator $Y \in \mathbb{B}(\ell_2(\{0, 1, \cdots, m\}))$,
$\,^l Y_k$ is the coefficient 
$\langle Y \delta_k, \delta_l \rangle$.
\item
We write $Y_k$ for the $k$-th column vector  
$(\,^l Y_k)_{l = 0}^m$ of $Y$.
\item
We use the same notations for an operator $Y \in \mathbb{B}(\ell_2(\mathbb{N}))$.
That is, $\,^l Y_k$ stands for 
$\langle Y \delta_k, \delta_l \rangle$ and
$Y_k$ stands for $(\,^l Y_k)_{l = 0}^\infty \in \ell_2(\mathbb{N})$.
\item
$
\displaystyle{
\frac{a_1}{b_1}
\begin{array}{c} \\ + \end{array}
\frac{a_2}{b_2}
\begin{array}{ccc}
  &           &  \\ 
+& ... &+
\end{array}
\frac{a_m}{b_m}}$
is the continued fraction 
$
\displaystyle{
\frac{a_1}{b_1 +
\displaystyle{\frac{a_2}{b_2 + 
\begin{array}{cc}
\ddots\\
b_{m -1} + \displaystyle{\frac{a_m}{b_m
}}
\end{array}
}}
}_.}$
\end{itemize}

\section{Fluctuation property of the Stieltjes transform at $i$}\label{sectionfluctuation}
The three-term recurrence relation (equation (\ref{equation3term})) give a bijective correspondence between
the moment sequences of symmetric probability measures
\begin{eqnarray*}
\{ \{M_n\}_{n = 0}^\infty ;
\exists \textrm{ prob. } \mu, 
d \mu(x) = d\mu(-x), M_n = \int_\mathbb{R} x^n d \mu,
\sharp(\mathrm{supp} \mu) = \infty
\},
\end{eqnarray*}
and the infinite Jacobi sequences 
$\{\{\omega_n \}_{n = 1}^\infty ; \omega_n > 0\}$.
For an infinite Jacobi sequence $\{\omega_n\}_{n = 1}^\infty$,
denote by $X$ the infinite matrix corresponding to
the sequence:
\begin{eqnarray*}
X =
\left(
\begin{array}{cccccccc}
0             & \sqrt{\omega_1} & 0              & \cdots \\
\sqrt{\omega_1} & 0             & \sqrt{\omega_2}  &            \\
0             & \sqrt{\omega_2} & 0              & \ddots \\
\vdots     &                & \ddots      & \ddots \\
\end{array}
\right).
\end{eqnarray*}
We regard $X$ as an operator defined on $\oplus_{\mathbb{N}} \mathbb{C} \subset \ell_2(\mathbb{N})$.
The operator $X$ is a densely defined symmetric operator on the Hilbert space $\ell_2(\mathbb{N})$.
We note that $X$ is not essentially self-adjoint in general and that $X^*$ is not necessarily self-adjoint.
In the case that the growth rate of the Jacobi sequence $\{\omega_n\}$ is large,
there exist more than one self-adjoint operators between $X$ and $X^*$.
We show the existence of these operators, by exploiting finite dimensional approximations of the operator $X$.

Let $X_m$ denote the $(m + 1) \times (m +1)$ matrix \begin{eqnarray*}
X_m =
\left(
\begin{array}{cccccccc}
0             & \sqrt{\omega_1} & 0              & \cdots  & 0              \\
\sqrt{\omega_1} & 0             & \sqrt{\omega_2}  & \ddots  & \vdots      \\
0             & \sqrt{\omega_2} & 0              & \ddots  & 0             \\
\vdots     &  \ddots   & \ddots      & \ddots  & \sqrt{\omega_m} \\
0             & \cdots     & 0              & \sqrt{\omega_m} & 0
\end{array}
\right)
\end{eqnarray*}
corresponding to the finite Jacobi sequence $\{\omega_1, \omega_2, \cdots, \omega_m\}$.
We regard $X_m$ as an operator acting on the Hilbert space
$\ell_2(\{0, 1, \cdots , m\})$.

We begin with investigating the Stieltjes transform $G_{m}(i)$ of $X_m$ at $i$ with respect to the state $\langle \cdot \delta_0, \delta_0 \rangle$. 
The complex number $G_{m}(i)$ is equal to the upper-left matrix coefficient of $(i - X_m)^{-1}$.
This quantity is expressed by the continued fraction
\footnote{This equality has been shown in many references. See, e.g., Hora-Obata \cite[Lemma 1.87]{HoraObata}.
This is also shown in the first equation of Lemma \ref{Lemma0Y}.}
\begin{eqnarray*}
\frac{1}{i}
\begin{array}{c} \\ -\end{array}
\frac{\omega_1}{i}
\begin{array}{c} \\ -\end{array}
\frac{\omega_2}{i}
\begin{array}{ccc}
  &           &  \\ 
-& \cdots &-
\end{array}
\frac{\omega_m}{i}.
\end{eqnarray*}
The quantity $i G_{m}(i)$ is simply expressed as
$\displaystyle{\frac{1}{1}
\begin{array}{c} \\ + \end{array}
\frac{\omega_1}{1}
\begin{array}{c} \\ + \end{array}
\frac{\omega_2}{1}
\begin{array}{ccc}
   &           &  \\ 
+ & \cdots & + 
\end{array}
\frac{\omega_m}{1}}.$
To study 
$\{i G_{m}(i)\}_{m =1}^\infty$,
we define sequences $\{A_m\}_{m =-1}^\infty$ and $\{B_m\}_{m =-1}^\infty$ by 
\begin{eqnarray*}
A_{-1} = 0, &A_0 = 1, &A_m = A_{m-1} + \omega_m A_{m - 2}, \\
B_{-1} = 1, &B_0 = 1, &B_m = B_{m-1} + \omega_m B_{m - 2}.
\end{eqnarray*}

\begin{lemma}\label{LemmaFraction}
The Stieltjes transform $G_m(i)$ of $X_m$ is expressed by
\begin{eqnarray*}
iG_{m}(i) = 
\frac{A_m}{B_m}.
\end{eqnarray*}
\end{lemma}
\begin{proof}
Define $A_{-2}$ and $B_{-2}$ by $A_{-2} = 1$, $B_{-2} = 0$.
Define $\omega_0$ by $1$.
Applying the equation (1.71) of Hora--Obata \cite{HoraObata}, 
we obtain the proposition. 
\end{proof}

\begin{lemma}\label{LemmaCommutationRelation}
For $m \geqq 1$, $A_m B_{m-1} - B_m A_{m-1} 
= (-1)^m \omega_1 \omega_2 \cdots \omega_m$.
\end{lemma}

\begin{proof}
We prove by induction.
For the case of $m = 1$,
we have
\[
A_1 B_0 - B_1 A_0 = 1 \cdot 1 - (1 + \omega_1) \cdot 1 = - \omega_1.
\]
Assume that the lemma holds.
Then we have
\begin{eqnarray*}
A_{m + 1} B_m - A_m B_{m + 1} 
&=&
(A_m + \omega_{m + 1} A_{m - 1}) B_m -
A_m (B_m + \omega_{m + 1} B_{m - 1})\\
&=& 
- \omega_{m + 1} (A_m B_{m-1} - B_m A_{m-1})\\
& =& (-1)^{m + 1} \omega_1 \cdots \omega_m \omega_{m + 1}.
\end{eqnarray*}\end{proof}

\begin{lemma}\label{LemmaPropositionDifference}
For $m \geqq 2$, $i G_{m}(i) - i G_{m-1} (i) =
\displaystyle{
\frac{(-1)^m \omega_1 \omega_2 \cdots \omega_m}
{B_{m - 1} B_m}
}
$
\end{lemma}
\begin{proof}
By Lemma \ref{LemmaFraction},
we have
\[
i G_{m}(i) - i G_{m-1} (i) =
\frac{A_m}{B_m} - \frac{A_{m-1}}{B_{m-1}}
=
\frac{A_m B_{m-1} - A_{m-1} B_m}{B_m B_{m-1}}.
\]
Lemma \ref{LemmaCommutationRelation}
yields the desired equation.
\end{proof}
It follows that the sequence $\{i G_m(i)\}_{m = 1}^\infty$ alternately increases and decreases.
More precisely,
\begin{itemize}
\item
If $m$ is odd,
then 
$i G_{m-1} (i) > i G_{m}(i)$;
\item
If $m$ is even,
then 
$i G_{m-1} (i) < i G_{m}(i)$.
\end{itemize}
We next investigate the difference $| i G_{m} (i) - i G_{m -1}(i) |$.
\begin{proposition}
For $m \geqq 3$, $\displaystyle{
\frac{|i G_{m} (i) - i G_{m -1}(i)|}
{|i G_{m-1} (i) - i G_{m -2}(i)|}
=
1 - \frac{B_{m-1}}{B_m} < 1
}$.
\end{proposition}
\begin{proof}
By Lemma \ref{LemmaPropositionDifference},
we have
\[
\frac{|i G_{m} (i) - i G_{m -1}(i)|}
{|i G_{m-1} (i) - i G_{m -2}(i)|}
=
\frac{\omega_1 \cdots \omega_m}{B_{m - 1} B_m}
\frac{B_{m - 2} B_{m - 1}}{\omega_1 \cdots \omega_{m - 1}}
=
\left| \frac{\omega_m B_{m -2}}{B_m}\right|
\]
By the definition of $B_m$,
the above quantity is equal to $(B_m - B_{m -1})/B_m$.
\end{proof}

The above proposition implies the following properties 
of $\{ i G_m(i) \}_{m = 1}^\infty$.
\begin{theorem}\label{TheoremSC}
\begin{itemize}
\item
The subsequence $\{i G_m(i) ; m \mathrm{\ even}\}$ strictly decreases.
\item
The subsequence $\{i G_m(i) ; m \mathrm{\ odd}\}$ strictly increases.
\item
If $n$ is an even natural number, $i G_n(i)$ is an upper bound of
$\{i G_m(i) ; m \mathrm{\ odd}\}$. 
\item
If $n$ is an odd natural number, $i G_n(i)$ is a lower bound of
$\{i G_m(i) ; m \mathrm{\ even}\}$. 
\end{itemize}
\end{theorem}
Define two complex numbers 
$G_\mathrm{even}$ and $G_\mathrm{odd}$ by
\begin{eqnarray*}
i G_\mathrm{even} = 
\lim_{m \rightarrow \infty} i G_{2m}(i), \qquad
i G_\mathrm{odd}   =
\lim_{m \rightarrow \infty} i G_{2m+1}(i).
\end{eqnarray*}

\begin{proposition}\label{PropositionSC}
The following conditions are equivalent:
\begin{enumerate}
\item\label{ConditionEvenOdd}
$G_\mathrm{even} \neq G_\mathrm{odd}$;
\item\label{ConditionProd}
$\displaystyle{
\lim_{m \rightarrow \infty} \prod_{n = 1}^m
\left(1 - \frac{B_{n-1}}{B_n}\right) > 0
}$;
\item\label{ConditionSum}
$\displaystyle{
\lim_{m \rightarrow \infty}\sum_{n = 1}^m
\frac{B_{n-1}}{B_n} < \infty
}$.
\end{enumerate}
\end{proposition}
\begin{proof}
Since $B_0 = 1$, $B_1 = 1 + \omega_1$ and $B_2 = 1 + \omega_1 + \omega_2$, we have
\begin{eqnarray*}
i G_2 (i) - i G_1 (i)
&=& 
\frac{1}{1 + \omega_1/(1 + \omega_2)}
-
\frac{1}{1 + \omega_1}\\
&=& 
\frac{\omega_1}{1 + \omega_1}
\frac{\omega_2}{1 + \omega_1 + \omega_2}\\
&=& 
\left(1 - \frac{B_0}{B_1}\right)
\left(1 - \frac{B_1}{B_2}\right).
\end{eqnarray*}
Equivalence between (\ref{ConditionEvenOdd}) and (\ref{ConditionProd}) follows from the equality
\begin{eqnarray*}
| i G_{m} (i) - i G_{m -1}(i) |
= (i G_2 (i) - i G_1 (i)) \prod_{n = 3}^m
\left(1 - \frac{B_{n-1}}{B_n}\right)
=
\prod_{n = 1}^m
\left(1 - \frac{B_{n-1}}{B_n}\right).
\end{eqnarray*}
Equivalence between (\ref{ConditionProd}) and (\ref{ConditionSum}) is well-known.
\footnote{Use the inequality
$- 2x < \log(1 - x) < -x, \quad 0 < x \leqq 1/2$.} 
\end{proof}

\begin{definition}
If one of the equivalent conditions in Proposition \ref{PropositionSC} holds,
then the infinite Jacobi sequence 
$\{\omega_m \}_{m = 1}^\infty$ 
is said to have property (SC).
\footnote{Property (SC) means that the subsequences of the Stieltjes transform
`Separately Converge.'}
\end{definition}

\section{Two self-adjoint operators with different Stieltjes transforms}
\label{sectionselfadjoint}

We are going to prove that if property (SC) holds,
then 
\begin{itemize}
\item
The subsequences $\{X_{2m}\}_{m = 0}^\infty$ and 
$\{X_{2m + 1}\}_{m = 0}^\infty$
of operators respectively converge in some sense
(Proposition \ref{PropositionSCSC});
\item
They provide two self-adjoint operators;
\item
Their spectral decompositions give two probability measures which realize
$\{\omega_m\}_{m = 1}^\infty$ as the Jacobi sequence.
\end{itemize}

Let $\{C_m\}_{m = 1}^\infty$ denote the sequence defined by
\begin{eqnarray*}
C_0 &=& 1,\\
C_1 &=& \frac{1}{1 + \omega_1},\\
C_m &=&
\frac{1}{1 + 
\displaystyle{
\frac{\omega_m}{ 
\begin{array}{c}
\ddots\\
\displaystyle{1+ \frac{\omega_2}{1 + \omega_1}}
\end{array}}}
}
= 
\frac{1}{1} 
\begin{array}{cc}\\+\end{array}
\frac{\omega_m}{1} 
\begin{array}{cc}\\+\end{array}
\frac{\omega_{m-1}}{1} 
\begin{array}{cc}\\+\end{array}
\ldots 
\begin{array}{cc}\\+\end{array}
\frac{\omega_1}{1}. 
\end{eqnarray*}

We note that the sequence $C_m$ satisfies the equation 
\begin{eqnarray}\label{EquationC}
C_{m + 1} = \frac{1}{1 + \omega_{m + 1} C_m}.
\end{eqnarray}

The sequence $C_m$ is expressed by $\{B_m\}_{m = 1}^\infty$.
\begin{lemma}
For $m \geqq 1$,
$C_m = B_{m-1} / B_m$.
\end{lemma}

As a consequence, 
property (SC) is equivalent to 
$\sum_{m = 1}^\infty C_m < \infty$.

\begin{proof}
We are going to prove by induction.
For the case that $m = 1$, 
we have
$\displaystyle{C_1 = \frac{1}{1 + \omega_1} = \frac{B_0}{B_1}}$.
Suppose that the lemma holds for $m$.
Then we have
\begin{eqnarray*}
C_{m + 1} 
= 
\frac{1}{1 + \omega_{m + 1}C_m}
= 
\frac{1}{1 + \omega_{m + 1} B_{m - 1} / B_m} 
= 
\frac{B_m}{B_m + \omega_{m + 1} B_{m - 1}} 
=
\frac{B_m}{B_{m + 1}}
\end{eqnarray*}
\end{proof}

We define $D^{(m)}_n$ by 
$\displaystyle{
\frac{\omega_{n + 1}}{1}
\begin{array}{ccc}
   &           &  \\ 
+ & ... & + 
\end{array}
\frac{\omega_m}{1}}
$.
For $n = 0, 1, \cdots, m-1$,
we have 
\begin{eqnarray}\label{EquationD}
D^{(m)}_n = \frac{\omega_{n + 1}}{1 + D^{(m)}_{n + 1}}.
\end{eqnarray}

\begin{lemma}\label{LemmaSeqD}
Let $n$ be an arbitrary natural number. 
Fix $n$.
Then $\{D^{(2m)}_n ; 2m > n\}$
and
$\{D^{(2m + 1)}_n ; 2m + 1 > n\}$ respectively converge to positive real numbers.
\end{lemma}
\begin{proof}
We first prove the lemma in the case of $n = 0$.
Since $i G_m(i)$ is equal to
$1 / (1 + D^{(m)}_0)$,
we have $D^{(m)}_0 = 1 /(i G_m(i)) - 1$.
From the fluctuation phenomenon, $i G_m(i)$ is in the closed interval 
\[
[i G_1(i), i G_2(i)]
=
\left[
\frac{1}{1}
\begin{array}{c}\\+\end{array}
\frac{\omega_1}{1}
, 
\frac{1}{1}
\begin{array}{c}\\+\end{array}
\frac{\omega_1}{1}
\begin{array}{c}\\+\end{array}
\frac{\omega_2}{1}
\right].\]
It follows that
$D^{(m)}_0$ is in the closed interval
$\displaystyle{\left[
\frac{\omega_1}{1 + \omega_2}
, 
\omega_1
\right]}$.
By Theorem \ref{TheoremSC},
the sequences $\{D^{(m)}_0 ; m \mathrm{\ is\ even}\}$
and
$\{D^{(m)}_0 ; m \mathrm{\ is\ odd}\}$ converge to real numbers.

Suppose that
the lemma holds for a natural number $n$.
The sequence $\{D^{(m)}_{n + 1}\}_{m = n + 2}^\infty$
is described as 
$D^{(m)}_{n + 1} = \omega_{n + 1}/D^{(m)}_n - 1$.
Since 
\[\displaystyle{D^{(m)}_n =
\frac{\omega_{n + 1}}{1}  
\begin{array}{c} \\ + \end{array}
\frac{\omega_{n + 2}}{1}  
\begin{array}{ccc}
  &         & \\ 
+&\cdots&+
\end{array}
\frac{\omega_{m}}{1}  
}\]
is in the closed interval
$\displaystyle{\left[
\frac{\omega_{n + 1}}{1}
\begin{array}{c}\\+\end{array}
\frac{\omega_{n + 2}}{1}
, 
\frac{\omega_{n + 1}}{1}
\begin{array}{c}\\+\end{array}
\frac{\omega_{n + 2}}{1}
\begin{array}{c}\\+\end{array}
\frac{\omega_{n + 3}}{1}
\right]}$,
we have 
\[D^{(m)}_{n + 1} = 
\frac{\omega_{n + 1}}{D^{(m)}_n} - 1\in
\left[
\frac{\omega_{n + 2}}{1 + \omega_{n + 3}}
, \omega_{n + 2}\right].\]
By the hypothesis of induction, the above lemma for $n + 1$ follows.
\end{proof}

We denote by $Y^{(m)}$ the matrix
$(i - X_m)^{-1} \in \mathbb{B}(\ell_2(\{0, 1, \cdots, m\}))$.
We decompose $Y^{(m)}$ into column vectors as
$Y^{(m)} = [Y^{(m)}_0 Y^{(m)}_1 \cdots Y^{(m)}_m]$.
We describe the vector $Y^{(m)}_0$ as
$\displaystyle{Y^{(m)}_0 = 
\left(
\begin{array}{c}
\,^0Y^{(m)}_0\\
\,^1Y^{(m)}_0\\
\vdots\\
\,^m Y^{(m)}_0
\end{array}
\right)
}$.
The vector $Y^{(m)}_0$ is described by $C_n$ and $D^{(m)}_n$ as follows.

\begin{lemma}\label{Lemma0Y}
\begin{eqnarray*}
\,^0Y^{(m)}_0 &=&
\displaystyle{
\frac{1}{i}} 
\frac{1}{
1 + D_0^{(m)}},\\
\,^n Y^{(m)}_0 &=& 
\displaystyle{
\frac{\sqrt{\omega_1 \cdots \omega_n} C_1 \cdots C_{n}}
{i^{n + 1}}}
\frac{1}{
1 +
C_n D_n^{(m)}}, \quad 1 \leqq n \leqq m -1,\\
\,^m Y^{(m)}_0 &=& 
\displaystyle{
\frac{\sqrt{\omega_1 \cdots \omega_m} C_1 \cdots C_{m}}
{i^{m + 1}}}.
\end{eqnarray*}
\end{lemma}

\begin{proof}
We define a vector $\displaystyle{\widetilde{Y} = 
\left(
\begin{array}{c}
\,^0\widetilde{Y}\\
\vdots\\
\,^m\widetilde{Y}
\end{array}
\right)
}$
by the right hand side of the equations which we are going to prove. By the injectivity of $(i - X_m)^{-1}$, it suffices to show that $\displaystyle{
\left(
\begin{array}{c}
\,^0Z\\
\vdots\\
\,^m Z
\end{array}
\right)
=
(i - X_m) \widetilde{Y}
}$ is equal to $\delta_0 = 
\left(
\begin{array}{c}
1\\
0\\
\vdots
\end{array}
\right)
$.
Equation (\ref{EquationC}) and (\ref{EquationD}) yield
\begin{eqnarray*}
\,^0Z &=&
i \left( \,^0 \widetilde{Y} \right) 
- \sqrt{\omega_1} \left( \,^1\widetilde{Y} \right)\\
&=&
\frac{1}{1 + D^{(m)}_0} + \frac{\omega_1 C_1}{1 + C_1 D^{(m)}_1}\\
&=&
\frac{1 + D^{(m)}_1}{1 + D^{(m)}_1 + \omega_{1}} 
+ \frac{\omega_1}{C_1^{-1} + D^{(m)}_1}\\
&=& 1.
\end{eqnarray*}
In order to prove  $^nZ = 0, 1 \leqq n \leqq m - 1$, we compute
\begin{eqnarray*}
&&- \sqrt{\omega_n} (\,^{n -1} \widetilde{Y})
+ i (\,^n\widetilde{Y})\\
&=&
\frac{\sqrt{\omega_1 \cdots \omega_n} C_1 \cdots C_{n-1}}{i^{n + 2}}
\left(
\frac{1}{1 + C_{n-1}D^{(m)}_{n-1}} 
-
\frac{C_n}{1 + C_{n}D^{(m)}_{n}} 
\right).
\end{eqnarray*}
By $D_{n -1}^{(m)} = \omega_n / (1 + D_n^{(m)})$
(equation (\ref{EquationD})) and 
$1 + \omega_n C_{n-1} = C_n^{-1}$
(equation (\ref{EquationC})),
we have
\begin{eqnarray*}
\frac{1}{1 + C_{n-1}D^{(m)}_{n-1}} 
-
\frac{C_n}{1 + C_{n}D^{(m)}_{n}} 
&=&
\frac{1 + D^{(m)}_n}{1 + \omega_{n} C_{n-1} + D^{(m)}_n} 
-
\frac{1}{C_{n}^{-1} + D^{(m)}_{n}} \\
&=&
\frac{D^{(m)}_n}{C_{n}^{-1} + D^{(m)}_n} \\
&=&
C_n \frac{1}{1 / D^{(m)}_{n} + C_n} \\
&=&
\omega_{n + 1} C_n 
\frac{1}{1 + \omega_{n + 1} C_{n} + D^{(m)}_{n + 1}}.
\end{eqnarray*}
Equation (\ref{EquationC})
for $n + 1$ yields
\begin{eqnarray*}
\frac{1}{1 + \omega_{n + 1} C_{n} + D^{(m)}_{n + 1}}
=
\frac{1}{1 /C_{n + 1} + D^{(m)}_{n + 1}}
=
\frac{C_{n + 1}}{1 + C_{n + 1} D^{(m)}_{n + 1}}.
\end{eqnarray*}
It follows that
\begin{eqnarray*}
^nZ &=&
- \sqrt{\omega_n} \left(\,^{n -1}\widetilde{Y}\right)
+ i \left(\,^n\widetilde{Y}\right)
- \sqrt{\omega_{n + 1}} \left(\,^{n + 1}\widetilde{Y}\right)\\
&=&
\frac{\sqrt{\omega_1 \cdots \omega_n} C_1 \cdots C_{n-1}}{i^{n + 2}}
\omega_{n + 1} C_n \frac{C_{n + 1}}{1 + C_{n + 1} D^{(m)}_{n + 1}}
- \sqrt{\omega_{n + 1}} \left(\,^{n + 1}\widetilde{Y}\right)\\
&=& 0.
\end{eqnarray*}
Since $D_{m - 1}^{(m)}$ is equal to $\omega_m$, 
we compute $^m Z$ as follows:
\begin{eqnarray*}
^m Z
&=&
-\sqrt{\omega_m} \left(\,^{m -1}\widetilde{Y}\right) 
+ i \left(\,^m \widetilde{Y}\right)\\
&=&
\frac{\sqrt{\omega_1 \cdots \omega_m} C_1 \cdots C_{m - 1}}
{i^{m}}
\left(
\frac{-1}{
1 +
C_{m - 1} D_{m - 1}^{(m)}}
+
C_m
\right)\\
&=& 0
\end{eqnarray*}
It follows that $(i - X_m) \widetilde{Y}$ is equal to $\delta_0$.
\end{proof}

We regard $Y^{(m)} = (i - X_m)^{-1} \in 
\mathbb{B}(\ell_2(\{0, 1, \cdots, m\}))$ as an operator
in $\mathbb{B}(\ell_2(\mathbb{N}))$.

\begin{lemma}\label{LemmaWeakConvergence}
The sequences 
$\{Y^{(2m)}\}_{m = 0}^\infty$
and
$\{Y^{(2m + 1)}\}_{m = 0}^\infty$
converge in the weak operator topology.
\end{lemma}

Recall that $\,^l Y^{(m)}_k$ stands for 
the $(l, k)$-entry of the matrix 
$Y^{(m)}$.

\begin{proof}
Note that the operator norm of $Y^{(m)}$ is at most $1$,
because $X_m$ is self-adjoint and $Y^{(m)}$ is the resolvent at $i$.
It suffices to show that
for arbitrary natural numbers $k, l \in \mathbb{N}$,
the sequences $\{\,^l Y^{(2m)}_k\}_{m = 0}^\infty$ and 
$\{\,^l Y^{(2m + 1)}_k\}_{m = 0}^\infty$ converge.
We prove this claim by induction on $k$.

By the concrete expressions of $\,^l Y_0^{(2m)}$ and
$\,^l Y_0^{(2m + 1)}$ in Lemma \ref{Lemma0Y} and by 
Lemma \ref{LemmaSeqD},
$\{\,^l Y_0^{(2m)}\}_m$ and
$\{\,^l Y_0^{(2m + 1)}\}_m$ respectively converge.
By the equation $Y^{(m)} (i - X_m) = 
\mathrm{id}_{\ell_2(\{0, 1, \cdots, m\})}$,
we have $i Y^{(m)}_0 - \sqrt{\omega_1} Y^{(m)}_1 
= \delta_0$.
Therefore the column vector $Y^{(m)}_1$ is equal to 
$(i Y^{(m)}_0 - \delta_0)/ \sqrt{\omega_1}$ 
and their entries converge.

Assume that the above claim holds for
the column vectors 
$Y^{(m)}_{k-2}$ and $Y^{(m)}_{k-1}$.
For the case of $2 \leqq k < m$,
observing the $(k -1)$-st column of the equation $Y^{(m)} (i - X_m) = 
\mathrm{id}_{\ell_2(\{0, 1, \cdots, m\})}$,
we have 
\begin{eqnarray*}
- \sqrt{\omega_{k-1}} Y^{(m)}_{k-2}
+ i Y^{(m)}_{k-1}
- \sqrt{\omega_{k}} Y^{(m)}_k &=& \delta_{k -1},
\\
Y^{(m)}_k
&=& 
\frac{- \sqrt{\omega_{k-1}} Y^{(m)}_{k-2}
+ i Y^{(m)}_{k-1} - \delta_{k -1}}
{\sqrt{\omega_{k}}}.
\end{eqnarray*}
By the assumption, $\{\,^{k}Y^{(2m)}_l\}_m$
and $\{\,^{k}Y^{(2m + 1)}_l\}_m$
respectively converge.
\end{proof}

In the case that property (SC) holds, we obtain a stronger conclusion in Proposition \ref{PropositionSCSC}.
Define an infinite sequence $\widehat{Y} =
\left(
\begin{array}{c}
\,^0\widehat{Y} \\
\,^1\widehat{Y} \\
\vdots
\end{array}
\right)
$
by
\[
\,^0\widehat{Y} = 1, 
\,^1\widehat{Y} = \sqrt{\omega_1}, 
\,^m\widehat{Y} = 
\sqrt{\omega_1 \cdots \omega_m} C_1 \cdots C_{m}.
\]
\begin{lemma}
The sequence $\widehat{Y}$ is an element of $\ell_2(\mathbb{N})$,
if property (SC) holds.
\end{lemma}

\begin{proof}
By the inequality
$\displaystyle{\omega_{n + 1} C_n C_{n + 1}
=
\omega_{n+1} C_n \frac{1}{1 + \omega_{n+1} C_n} 
\leqq 1}$,
we estimate
$(\,^m \widehat{Y})^2$
by
\begin{eqnarray*}
(\,^m \widehat{Y})^2 
=
\omega_1 C_1 
\left( \prod_{n = 1} ^{m -1}
\omega_{n + 1} C_n C_{n + 1} \right)
C_m 
\le
\omega_1 C_1 C_m.
\end{eqnarray*}
If $\{\omega_m\}_{m = 1}^\infty$ has property (SC), then
$
\|\widehat{Y}\|_2^2
\leqq 1 +
\omega_1 C_1 \sum_{m = 1}^\infty 
C_m
< \infty$.
\end{proof}

\begin{lemma}
If property (SC) holds,
then there exist
positive square summable sequences
$\{\widehat{Y}_k = 
( \,^l\widehat{Y}_k )_{l = 0}^\infty  ; k \in \mathbb{N} \}
\subset \ell_2(\mathbb{N})
$ such that
for every $k, l\in \mathbb{N}$ and $m \geqq k$,
$|\,^l Y^{(m)} _k|  \leqq \,^l\widehat{Y}_k$.
\end{lemma}

\begin{proof}
We prove by induction on $k$.
Define $\widehat{Y}_0$ by $\widehat{Y}$.
By Lemma \ref{Lemma0Y} and the definition of $\widehat{Y}$,
we have $|\,^l Y^{(m)} _0 |  \leqq \,^l \widehat{Y}_0$.
For $m \geqq 1$, the vector $Y^{(m)}_1$ is expressed by 
$(i Y^{(m)}_0 - \delta_0)/{\sqrt{\omega_1}}$.
Absolute values of entries are dominated by those of $(\widehat{Y}_0 + \delta_0) /{\sqrt{\omega_1}}$. Define $\widehat{Y}_1$ by 
$(\widehat{Y}_0 + \delta_0)/{\sqrt{\omega_1}} $.

For $m \geqq k \geqq 1$, 
since the $(k - 1)$-st column of $Y^{(m)} (i - X_m)$ is
\[ \delta_{k -1} = 
- \sqrt{\omega_{k - 1}} Y_{k - 2}^{(m)} 
+ i Y_{k -1}^{(m)}
- \sqrt{\omega_k} Y_k^{(m)},
\]
the column vector $Y^{(m)}_k$ is expressed by 
$(i Y^{(m)}_{k - 1} -
\sqrt{\omega_{k-1}} Y^{(m)}_{k - 2} - \delta_{k-1})/ \sqrt{\omega_k}$.
Absolute values of entries are dominated by those of 
\[
\widehat{Y}_k =
\frac{\widehat{Y}_{k - 1} +
\sqrt{\omega_{k-1}} \widehat{Y}_{k - 2} 
+ \delta_{k-1}} {\sqrt{\omega_k}}.
\]
Repeating this procedure, we obtain vectors 
$\{\widehat{Y}_k\}$ satisfying the above lemma.
\end{proof}

\begin{proposition}\label{PropositionSCSC}
If property (SC) holds, then the sequences 
$\{Y^{(2m)}\}_m \subset \mathbb{B}(\ell_2(\mathbb{N}))$ and 
$\{Y^{(2m + 1)}\}_m \subset \mathbb{B}(\ell_2(\mathbb{N}))$
respectively converge in the strong operator topology.
\footnote{Property (SC) is a `Sufficient Condition' 
for the `Strong Convergence.'}
\end{proposition}

\begin{proof}
Suppose that property (SC) holds.
For every $k \in \mathbb{N}$, 
the sequence of the $k$-th column vectors 
$\{Y^{(2m)}_k\}_{m}$ strongly converges. 
Indeed, by Lemma \ref{LemmaWeakConvergence}, 
for every $l \in \mathbb{N}$,
$\{\,^l Y^{(2m)}_k \}_{m}$ converges.
For $m \geqq k/2$,
the absolute value of $Y^{(2m)}_k$ dominated by a square summable positive sequence $\widehat{Y}_k$. 
By  Lebesgue's dominated convergence theorem,
$\{Y^{(2m)}_k \}_{m}$ 
strongly converges. 
By the same proof, for every $k$,
the sequence of vectors $\{Y^{(2m + 1)}_k \}_{m}$ 
also strongly converges. 

Since the operator norms of $Y^{(m)} = (i - X_m)^{-1}$ are at most $1$,
the sequences 
$\{Y^{(2m)}\}_{m = 0}^\infty$ and 
$\{Y^{(2m + 1)}\}_{m = 0}^\infty$
respectively converge in the strong operator topology.
\end{proof}

Now we are ready to construct two self-adjoint operators $X_\mathrm{even}$ and $X_\mathrm{odd}$,
in the case that $\{\omega_n\}$ holds property (SC).
Let $Y^\mathrm{even}$ denote the strong limit of $\{Y^{(2m)}\}_{m = 0}^\infty$
and 
let $Y^\mathrm{odd}$ denote the strong limit of 
$\{Y^{(2m + 1)}\}_{m = 0}^\infty$.
Recall that $X^*$ is the adjoint of the Jacobi matrix $X : \oplus_\mathbb{N} \mathbb{C} 
\rightarrow \oplus_\mathbb{N} \mathbb{C} \subset \ell_2(\mathbb{N})$.

\begin{lemma}\label{LemmaInv}
If property (SC) holds, then
\begin{eqnarray*}
\mathrm{Image}(Y^{\mathrm{even}})
\subset \mathrm{Dom}(X^*), &&
(i - X^*) Y^\mathrm{even} = \mathrm{id}_{\ell_2(\mathbb{N})},\\
\mathrm{Image}(Y^{\mathrm{odd}})
\subset \mathrm{Dom}(X^*), &&
(i - X^*) Y^\mathrm{odd} = \mathrm{id}_{\ell_2(\mathbb{N})}.
\end{eqnarray*}
\end{lemma}

\begin{proof}
We prove the first equality.
It suffices to show 
\begin{eqnarray*}
\langle Y^\mathrm{even} \xi, (- i - X) \eta\rangle
= 
\langle \xi, \eta \rangle,
\end{eqnarray*}
for $\xi \in \mathrm{Image}(Y^{\mathrm{even}})$,
and every finitely supported vector 
$\eta \in \oplus_\mathbb{N} \mathbb{C}$.
If $m$ is so large that $\{0, 1, \cdots, 2m - 1\}$ includes $\mathrm{supp}(\eta)$, 
then
\[
\langle Y^\mathrm{even} \xi, (- i - X) \eta\rangle
=
\langle Y^\mathrm{even} \xi, (- i - X_{2m}) \eta\rangle.\]
Therefore we have
\begin{eqnarray*}
\langle Y^\mathrm{even} \xi, (- i - X) \eta\rangle
&=& 
\lim_m 
\langle Y^\mathrm{(2m)} \xi, (- i - X_{2m}) \eta\rangle\\
&=& 
\lim_m 
\langle (i - X_{2m}) Y^\mathrm{(2m)} \xi,  \eta\rangle\\
&=& 
\lim_m 
\langle \xi,  P_{2m} \eta \rangle,
\end{eqnarray*}
where $P_{2m}$ stands for the orthogonal projection onto 
$\ell_2(\{0, 1, \cdots, 2m\})$.
It follows that 
$\langle Y^\mathrm{even} \xi, (- i - X) \eta\rangle
= 
\langle \xi, \eta\rangle$.
\end{proof}

\begin{lemma}
\label{LemmaImages}
If property (SC) holds,
then 
\[
\mathrm{Image}(Y^\mathrm{even})^* = \mathrm{Image}(Y^\mathrm{even}), \quad
\mathrm{Image}(Y^\mathrm{odd})^* = \mathrm{Image}(Y^\mathrm{odd}).
\]
\end{lemma}
\begin{proof}
Since the matrix coefficients of $i - X_m$ are symmetric,
those of $Y^{(m)}$ are also symmetric.
More precisely, $\,^k Y^{(m)}_l$ is equal to $\,^l Y^{(m)}_k$.
It follows that the adjoint $(Y^{(m)})^*$ is equal to $\overline{Y^{(m)}} = \left( \overline{\,^l Y^{(m)}_k} \right)_{l, k}$.

Notice that all the arguments in this paper is also valid,
even if we replace `$i$' with `$-i$.'
It follows that the sequences 
\[\{(Y^{(2m)})^* = \overline{Y^{(2m)}}\}_{m = 0}^\infty, \quad
\{(Y^{(2m + 1)})^* = \overline{Y^{(2m + 1)}}\}_{m = 0}^\infty\] strongly converge.
Observing their matrix coefficients, we have
\begin{eqnarray*}
\lim_m (Y^{(2m)})^* 
= \overline{Y^\mathrm{even}} 
= (Y^\mathrm{even})^*, \quad
\lim_m (Y^{(2m + 1)})^* 
= \overline{Y^\mathrm{odd}}
= (Y^\mathrm{odd})^*.
\end{eqnarray*}
By the resolvent identity, we have
\begin{eqnarray*}
(- i - X_{2m})^{-1} 
&=& (i - X_{2m})^{-1} + 2i (i - X_{2m})^{-1} ( - i - X_{2m})^{-1},\\
(Y^{(2m)})^* 
&=& Y^{(2m)} \left( 1 + 2i \overline{Y^{(2m)}} \right)
= Y^{(2m)} (1 + 2i (Y^{(2m)})^*).
\end{eqnarray*}
Taking limits with respect to the strong operator topology, we have
\begin{eqnarray*}
(Y^\mathrm{even})^* 
= Y^\mathrm{even} (1 + 2i (Y^\mathrm{even})^*).
\end{eqnarray*}
It follows that the image of $(Y^\mathrm{even})^*$ is
included in that of $Y^\mathrm{even}$.
Replacing `$i$' with `$-i$,' we have
\begin{eqnarray*}
\mathrm{Image}(Y^\mathrm{even})^* 
= \mathrm{Image}Y^\mathrm{even}.
\end{eqnarray*}
By the same proof, we also obtain 
the equality $\mathrm{Image}(Y^\mathrm{odd})^* 
= \mathrm{Image}Y^\mathrm{odd}$.
\end{proof}

\begin{lemma}
\label{LemmaInverseAdjoint}
Let $Y$ be a bounded operator on a Hilbert space $\mathcal{H}$.
Suppose that $Y$ and its adjoint $Y^*$ are injective.
Then the images of $Y$, $Y^*$ are dense in $\mathcal{H}$. 
The adjoint of $Y^{-1} \colon \mathrm{Image}(Y) \rightarrow \mathcal{H}$ is equal to the inverse of $Y^*$.
\end{lemma}

\begin{proof}
By the equalities 
$\overline{\mathrm{Image}(Y)} = \mathrm{Ker}(Y^*)^\perp$
and 
$\overline{\mathrm{Image}(Y^*)} = \mathrm{Ker}(Y)^\perp$, 
we obtain the first assertion.

For a closed operator $Z$ on $\mathcal{H}$, let $\mathcal{G}(Z)$ denote the graph 
\[\{\xi \oplus Z \xi ; \xi \in \mathrm{Dom}(Z)\} \subset \mathcal{H} \oplus \mathcal{H}.\]
We define two unitary operators $U$, $V$ on $\mathcal{H} \oplus \mathcal{H}$ by
\begin{eqnarray*}
V(\xi \oplus \eta) = (- \eta) \oplus \xi, \quad
U(\xi \oplus \eta) = \eta \oplus \xi.
\end{eqnarray*}
Then we have $\mathcal{G}(Z^*) = (V \mathcal{G}(Z))^\perp$.
When $Z$ is injective, the graph of 
\[
Z^{-1} \colon \mathrm{Image}(Z) \rightarrow \mathrm{Dom}(Z)
\] 
is given by
$\mathcal{G}(Z^{-1}) = U \mathcal{G}(Z)$.

The graph of $(Y^*)^{-1}$ is given by
$U((V \mathcal{G}(Y))^\perp)$.
The graph of $(Y^{-1})^{*}$ is given by
$(V U\mathcal{G}(Y))^\perp$.
Since $- U V = V U$, we obtain 
$\mathcal{G}((Y^*)^{-1}) = \mathcal{G}((Y^{-1})^{*})$. 
\end{proof}

We define $X_\mathrm{even}$ and $X_\mathrm{odd}$ by
\[X_\mathrm{even} = X^* |_{\mathrm{Image}Y^\mathrm{even}},
X_\mathrm{odd} = X^* |_{\mathrm{Image}Y^\mathrm{odd}}.\]

\begin{proposition}
If the Jacobi sequence $\{\omega_m\}_{m = 1}^\infty$ has property (SC),
then 
\begin{itemize}
\item
$X_\mathrm{even}$ and $X_\mathrm{odd}$ are extensions of $X$;
\item
$X_\mathrm{even}$ and $X_\mathrm{odd}$ are self-adjoint operators;
\item
$X_\mathrm{even}$ and $X_\mathrm{odd}$ have finite moments of all orders $m$ whose Jacobi sequence is identical to $\{\omega_m\}_{m = 1}^\infty$;
\item
Their Stieltjes transforms $\langle (i - X_\mathrm{even})^{-1} \delta_0, \delta_0 \rangle$, $\langle (i - X_\mathrm{odd})^{-1} \delta_0, \delta_0 \rangle$ at $i$ are different.
\end{itemize}
\end{proposition}

In the proof, denote the anti-linear isometry $J \colon \ell_2(\mathbb{N}) \rightarrow \ell_2(\mathbb{N}) $ defined by
$\,^l(J \eta) = \overline{\,^l \eta}$.
Since $J X J = X$, we have $J X^* J = X^*$.
We also have 
\begin{eqnarray*}
J Y^{(m)} J 
= J (i - X_m)^{-1} J
= J (- i - X_m)^{-1} J 
= ((i - X_m)^*)^{-1}
= (Y^{(m)})^{*}.
\end{eqnarray*}
Taking limits with respect to the strong operator topology,
we have 
\[J Y^\mathrm{even} J = (Y^\mathrm{even})^*, \quad
J Y^\mathrm{odd} J = (Y^\mathrm{odd})^*.\]

\begin{proof}
Note that the operator $Y^\mathrm{even} : \ell_2(\mathbb{N}) \rightarrow \mathrm{Image} Y^\mathrm{even}$ is injective by Lemma \ref{LemmaInv} and that
the operator $(Y^\mathrm{even})^* = J Y^\mathrm{even} J$ is also injective.
Since $i - X_\mathrm{even}$ is equal to $(Y^\mathrm{even})^{-1}$,
Lemma \ref{LemmaInverseAdjoint} yields that
\begin{eqnarray*}
(i - X_\mathrm{even})^*
=
((Y^\mathrm{even})^*)^{-1}.
\end{eqnarray*} 
By Lemma \ref{LemmaInv}, 
we have
$(- i - X^*) (Y^\mathrm{even})^* =
J (i - X^*) Y^\mathrm{even} J = \mathrm{id}_{\ell_2(\mathbb{N})}$.
This means that
\[((Y^\mathrm{even})^{*})^{-1} = (- i - X^*) |_{\mathrm{Image}(Y^\mathrm{even})^{*}}.\]
By Lemma \ref{LemmaImages},
we have
\[
(- i - X^*) |_{\mathrm{Image}(Y^\mathrm{even})^{*}} = (- i - X^*) |_{\mathrm{Image}Y^\mathrm{even}}
= - i - X_\mathrm{even}.
\]
It follows that
$(i - X_\mathrm{even})^* = - i - X_\mathrm{even}$.
We conclude that $X_\mathrm{even}$ is self-adjoint.
The same proof works for $X_\mathrm{odd}$.

By taking adjoints of the operators 
$X_\mathrm{even} \subset X^*$,
we obtain the inclusion $X \subset \overline{X} 
\subset X_\mathrm{even}$.
We also have $X \subset X_\mathrm{odd}$.
These operators have the same moments with respect to the state $\langle \cdot \delta_0, \delta_0 \rangle$.
The Stieltjes transform of $X_\mathrm{even}$ at $i$ is equal to
\begin{eqnarray*}
\langle Y^\mathrm{even} \delta_0, 
\delta_0 \rangle 
= 
\lim_{m} \langle Y^{(2m)} \delta_0, 
\delta_0 \rangle
=
\lim_{m} i G_{2m}(i)
=
i G_\mathrm{even}.
\end{eqnarray*}
The Stieltjes transform of $X_\mathrm{odd}$ at $i$ is equal to
$i G_\mathrm{odd}.
$
Since property (SC) holds, these values are different (see Proposition \ref{PropositionSC}).
\end{proof}

\begin{theorem}\label{TheoremTwoMeasures}
If the Jacobi sequence $\{\omega_n\}_{n = 0}^\infty$ has property (SC),
then there exist two probability measures 
$\mu_\mathrm{even}$ and $\mu_\mathrm{odd}$ such that
\begin{itemize}
\item
These measures $\mu_\mathrm{even}$ and $\mu_\mathrm{odd}$ have finite moments of all orders $n$
whose Jacobi sequence is identical to $\omega_n$;
\item
These measures are symmetric under the reflexion $\mathbb{R} \ni x \mapsto -x \in \mathbb{R}$;
\item
Their Stieltjes transforms at $i$ are different.
That is, 
\begin{eqnarray*}
\int_{x \in \mathbb{R}} \frac{1}{i - x} d \mu_\mathrm{even}
\neq
\int_{x \in \mathbb{R}} \frac{1}{i - x} d \mu_\mathrm{odd}.
\end{eqnarray*}
\end{itemize}
\end{theorem}

\begin{proof}
Let $\int_\mathbb{R} x d E_\mathrm{even}$ be the spectral decomposition of $X_\mathrm{even}$ and
let $\int_\mathbb{R} x d E_\mathrm{odd}$ be the spectral decomposition of $X_\mathrm{odd}$.
Define a probability measure $\mu_\mathrm{even}$ by 
$\langle E_\mathrm{even}(\cdot) \delta_0, \delta_0 \rangle$
and define $\mu_\mathrm{odd}$ by 
$\langle E_\mathrm{odd}(\cdot) \delta_0, \delta_0 \rangle$.
The measures $\mu_\mathrm{even}$ and $\mu_\mathrm{odd}$
satisfy
the first and the third assertions.

Let $P(z, \overline{z}) \in \mathbb{C}[z, \overline{z}]$ 
be an arbitrary commutative polynomial.
Let $\mu_m$ denote the probability measure corresponding to $X_m$ and the state $\langle \cdot \delta_0, \delta_0 \rangle$.
Define $p(x)$ by $P((i - x)^{-1}, (- i - x)^{-1})$.
Note that $p(x)$ is equal to some polynomial function on the support
of $\mu_m$.
Since the moments of odd orders 
$\langle X_m^{2n + 1} \delta_0, \delta_0 \rangle$ are zero,
we have
\begin{eqnarray*}
\int_\mathbb{R} p(x) d \mu_m
=
\int_\mathbb{R} p(-x) d \mu_m.
\end{eqnarray*}
The left hand side is equal to
\begin{eqnarray*}
\int_\mathbb{R} P((i - x)^{-1}, (- i - x)^{-1}) d \mu_m
=
\langle P(Y^{(m)}, (Y^{(m)})^{*}) \delta_0, \delta_0 \rangle.
\end{eqnarray*}
The right hand side is equal to
\begin{eqnarray*}
\int_\mathbb{R} P((i + x)^{-1}, (- i + x)^{-1}) d \mu_m
&=&
\int_\mathbb{R} P(- (- i - x)^{-1}, - (i - x)^{-1}) d \mu_m\\
&=&
\langle P(- (Y^{(m)})^{*}, - Y^{(m)}) \delta_0, \delta_0 \rangle.
\end{eqnarray*}
Thus we obtain
$\langle P(Y^{(m)}, (Y^{(m)})^{*}) \delta_0, \delta_0 \rangle
=
\langle P(- (Y^{(m)})^{*}, - Y^{(m)}) \delta_0, \delta_0 \rangle$.
Taking a limit with respect to the strong operator topology,
we obtain
\[
\langle P(Y^\mathrm{even}, (Y^\mathrm{even})^{*}) \delta_0, \delta_0 \rangle
=
\langle P(- (Y^\mathrm{even})^{*}, - Y^\mathrm{even}) \delta_0, \delta_0 \rangle,
\]
and hence
\begin{eqnarray*}
\int_\mathbb{R} P((i - x)^{-1}, (- i - x)^{-1})
 d \mu_\mathrm{even}
&=&
\int_\mathbb{R} P((i + x)^{-1}, (- i + x)^{-1})
 d \mu_\mathrm{even},\\
\int_\mathbb{R} p(x) d \mu_\mathrm{even}
&=&
\int_\mathbb{R} p(-x) d \mu_\mathrm{even}.
\end{eqnarray*}
Since the polynomials $\{P((i - x)^{-1}, (- i - x)^{-1})\}$ form a dense subspace of the C$^*$-algebra 
$C_0(\mathbb{R})$, we have 
$d \mu_\mathrm{even} (x) = d \mu_\mathrm{even} (-x)$.
The measure $\mu_\mathrm{odd}$ is also symmetric.
\end{proof}

\section{Examples of infinite Jacobi sequences 
with property (SC)}
\label{sectionexamples}

When we seek examples of property (SC), the following are useful.
\begin{lemma}\label{LemmaEstimateC}
$\displaystyle{
C_n
< 
\frac{1}{\omega_n} +
\frac{\omega_{n - 1}}{\omega_n} C_{n - 2}}\quad n \geqq 2$.
\begin{eqnarray*}
&& C_{2n}
< 
\frac{1}{\omega_{2n}} +
\frac{\omega_{2n - 1}}{\omega_{2n}} 
\frac{1}{\omega_{2n - 2}} +
\cdots +
\left(
\frac{\omega_{2n - 1}}{\omega_{2n}} 
\cdots
\frac{\omega_{3}}{\omega_{4}}
\frac{1}{\omega_{2}} 
\right)
+
\left(
\frac{\omega_{2n - 1}}{\omega_{2n}} 
\cdots
\frac{\omega_{3}}{\omega_{4}}
\frac{\omega_1}{\omega_{2}}
\right),
\\
&& C_{2n + 1}
< 
\frac{1}{\omega_{2n + 1}} +
\frac{\omega_{2n}}{\omega_{2n + 1}} 
\frac{1}{\omega_{2n - 1}} +
\cdots +
\left(
\frac{\omega_{2n}}{\omega_{2n + 1}} 
\cdots
\frac{\omega_{2}}{\omega_{3}}
\frac{1}{\omega_{1}} 
\right), \quad n \geqq 1.
\end{eqnarray*}
\end{lemma}

\begin{proof}
By the definition of $C_n$, we have
\begin{eqnarray*}
C_n
=
\frac{1}{1 + \displaystyle{\frac{\omega_n}{1 + \omega_{n - 1} C_{n - 2}}}}
< 
\frac{1 + \omega_{n - 1} C_{n - 2}}{\omega_n}
=
\frac{1}{\omega_n} +
\frac{\omega_{n - 1}}{\omega_n} C_{n - 2}.
\end{eqnarray*}
Applying this inequality to $C_{n-2}$, we have
\begin{eqnarray*}
C_n
&<&
\frac{1}{\omega_n} +
\frac{\omega_{n - 1}}{\omega_n} C_{n - 2}\\
&<&
\frac{1}{\omega_n} +
\frac{\omega_{n - 1}}{\omega_n} 
\left(
\frac{1}{\omega_{n - 2}} +
\frac{\omega_{n - 3}}{\omega_{n - 2}} C_{n - 4}
\right)\\
&=&
\frac{1}{\omega_n} +
\frac{\omega_{n - 1}}{\omega_n} 
\frac{1}{\omega_{n - 2}} +
\frac{\omega_{n - 1}}{\omega_n} 
\frac{\omega_{n - 3}}{\omega_{n - 2}} C_{n - 4}.
\end{eqnarray*}
Repeating this procedure, we obtain the second and the third inequalities.
\end{proof}

\subsection{Jacobi sequences with high growth rates}

We study the case that the infinite 
Jacobi sequence $\{\omega_n\}$
satisfies the following condition: 
\begin{itemize}
\item[]
$(*)$
There exist $0 < \alpha < 1$ and $n_0 \in \mathbb{N}$
such that $\omega_n < \alpha \omega_{n + 1}$ for $n \geqq n_0$.
\end{itemize}
This condition implies
that there exists a constant $K > 0$ such that
$1/ \omega_n < K \alpha^n$.

\begin{theorem}\label{TheoremHighGrowth}
The sequence $\{\omega_n\}$ with condition $(*)$ has property $(SC)$.
As a consequence, 
the corresponding probability measure is not unique.
\end{theorem}

\begin{proof}
Define an even number $2 m_0$ by $n_0$ or $n_0 -1$.
By the first inequality in Lemma \ref{LemmaEstimateC},
for $n \geqq n_0 + 1$,  we have
$C_n < K \alpha^n + \alpha C_{n - 2}$.
Applying this inequality many times, 
we obtain for $2n > n_0 + 1$,
\begin{eqnarray*}
&&C_{2n} \\
&<& 
K \alpha^{2n} + \alpha K \alpha^{2n - 2} 
+ \cdots 
+ \alpha^{(2n - 2 m_0)/2 - 1} K \alpha^{2 m_0 + 2}
 + \alpha^{(2n - 2 m_0)/2} C_{2m_0}\\
&<& 
\max \{K, C_{2m_0}\} \left( 
\alpha^{2n} + \alpha^{2n - 1} + \cdots + \alpha^{n + m_0 + 1} + \alpha^{n + m_0} \right)\\
&<& \max \{K, C_{2m_0}\} \cdot \alpha^n / (1 - \alpha).
\end{eqnarray*}
We obtain for $2n + 1 > n_0 + 1$,
\begin{eqnarray*}
&& C_{2n + 1} \\
&<& 
K \alpha^{2n + 1} + \cdots 
+ \alpha^{(2n - 2 m_0)/2 - 1} K \alpha^{2 m_0 + 3}
 + \alpha^{(2n - 2 m_0)/2} C_{2m_0 + 1}\\
&<& 
\max \{K, C_{2m_0 + 1}\} \left( 
\alpha^{2n + 1} + \alpha^{2n} + \cdots + \alpha^{n + m_0 + 1} + \alpha^{n + m_0 + 1} \right)\\
&<& \max \{K, C_{2m_0 + 1}\} \cdot \alpha^n / (1 - \alpha).
\end{eqnarray*}
It follows that
$\sum_{n = 1}^\infty C_n 
= \sum_{n = 1}^{n_0 + 1} C_n + \sum_{n = n_0 + 2}^\infty C_n 
< \infty$.
\end{proof}

\subsection{Jacobi sequences for $q$-Gaussians}
The Jacobi sequence of the Gaussian measure $\displaystyle{ e^{- x^2 /2} dx/ \sqrt{2 \pi}}$ is the sequence $\{n = [n]_1\}$ of integers.
In this subsection, we study the following deformations of the sequence:
\begin{itemize}
\item
$\displaystyle{
\omega_n = [n]_q = 1 + q + q^2 + \cdots + q^{n -1} 
= \frac{q^n - 1}{q - 1}, \quad q > - 1}$,
\item
$\displaystyle{
\omega_n 
= (-1)^{n-1} \frac{q^n - 1}{q - 1},
\quad
q < - 1}$.
\end{itemize}
The number $[n]_q$ is called a $q$-deformed integer.
Their corresponding probability measures are often called
``$q$-Gaussians.''
The study of $q$-Gaussians are motivated by the study of $q$-canonical commutation relations and $q$-canonical anti-commutation relations.
See \cite{Bozejko} and references therein. 

\begin{theorem}\label{qGauss}
The moment problem for $q$-Gaussians $(q\in\mathbb{R}, q \neq -1)$ are indeterminate if and only if $|q|>1$. 
\end{theorem}

\begin{proof}
If $|q| > 1$,
the Jacobi sequence satisfies condition $(*)$, since
\begin{eqnarray*}
\lim_n \frac{\omega_{n + 1}}{\omega_n}
=
\lim_n \left| \frac{q^{n + 1} -1}{q^n - 1}\right|
=
|q| > 1.
\end{eqnarray*}
By Theorem \ref{TheoremHighGrowth}, the Jacobi sequence has property (SC). By Theorem \ref{TheoremTwoMeasures},
there are more than two symmetric probability measures 
corresponding to $\omega_n$.
In the case that $-1 < q \leqq 1$, we have $[q]_n \leqq n$.
Carleman's condition states that
if $\{\omega_n\}$
satisfies the condition $\sum_{n} 1/\sqrt{\omega_n} = + \infty$, 
then the corresponding probability measure is unique. Theorem.\ref{equivalence} proved later also yields the results. 
\end{proof}

\begin{remark}
For $-1< q \leqq 1$, the theorem is well-known.
For $q > 1$, Ismail--Masson \cite{IsmailMasson} has already shown 
the indeterminacy.
\end{remark}
%which states that probability measure corresponding to
%$\{[n]_q\}_{n = 1}^\infty$ is not unique.

\subsection{Power of integers}
We are going to consider the case of
$\{\omega_n = n^p\}_{n = 1}^\infty$.

\begin{lemma}\label{LemmaEstimateSQRT}
If $\omega_n = n^p$ and $p \geqq 0$, then
\[ \frac{\omega_n}{\omega_{n + 1}} < \sqrt{\frac{\omega_n}{\omega_{n + 2}}} , \quad  
\frac{\sqrt{\omega_n}}{\omega_{n - 1}} <
2^p \frac{1}{\sqrt{\omega_n}} \quad (n \geqq 2).\]
\end{lemma}
\begin{proof}
By the inequality $n / (n + 1) < (n + 1) / (n + 2)$,
we have 
\[\sqrt{\omega_n / \omega_{n + 1}} 
<
\sqrt{\omega_{n + 1} / \omega_{n + 2}}. \]
Thus we obtain the following inequality:
\begin{eqnarray*}
\frac{\omega_n}{\omega_{n + 1}}
=
\sqrt{\frac{\omega_n}{\omega_{n + 1}}}
\sqrt{\frac{\omega_n}{\omega_{n + 1}}}
< 
\sqrt{\frac{\omega_n}{\omega_{n + 1}}}
\sqrt{\frac{\omega_{n + 1}}{\omega_{n + 2}}}
=
\sqrt{\frac{\omega_n}{\omega_{n + 2}}}.
\end{eqnarray*}
The second inequality is due to the following computation:
\begin{eqnarray*}
\frac{\sqrt{\omega_n}}{\omega_{n - 1}} 
=
\frac{\omega_n}{\omega_{n - 1}} 
\frac{1}{\sqrt{\omega_n}} 
<
\left(\frac{n}{n - 1}\right)^p \frac{1}{\sqrt{\omega_n}}
<
2^p \frac{1}{\sqrt{\omega_n}}.
\end{eqnarray*}
\end{proof}

\begin{proposition}\label{TheoremSCforPower}
For every $p > 2$, the Jacobi sequence 
$\{n^p\}_{n = 1}^\infty$
has property (SC).
As a consequence,
there exist more than two symmetric probability measures
whose Jacobi sequence is $\{n^p\}_{n = 1}^\infty$.
\end{proposition}

\begin{proof}
By Lemma \ref{LemmaEstimateSQRT},
we have
\begin{eqnarray*}
\frac{\omega_{n -1}}{\omega_n}
\frac{\omega_{n - 3}}{\omega_{n-2}} \cdots
\frac{\omega_{l}}{\omega_{l + 1}} 
\frac{1}{\omega_{l - 1}} 
&<&
\sqrt{\frac{\omega_{n -1}}{\omega_{n+1}}}
\sqrt{\frac{\omega_{n -3}}{\omega_{n- 1}}} \cdots
\sqrt{\frac{\omega_{l}}{\omega_{l + 2}}} 
\frac{1}{\omega_{l - 1}} \\
&=&
\frac{1}{\sqrt{\omega_{n + 1}}}
\frac{\sqrt{\omega_{l}}}{\omega_{l - 1}} \\ 
&<& 2^p 
\frac{1}{\sqrt{\omega_{n + 1}}}
\frac{1}{\sqrt{\omega_{l}}}.
\end{eqnarray*}
According to Lemma \ref{LemmaEstimateC},
the following inequalities hold:
\begin{eqnarray*}
&& C_{2n}
< 
2^p
\frac{1}{\sqrt{\omega_{2n + 1}}}
\left(
\frac{1}{\sqrt{\omega_{2n + 1}}} +
\frac{1}{\sqrt{\omega_{2n - 1}}} +
\cdots +
\frac{1}{\sqrt{\omega_{3}}} 
+
\frac{1}{\sqrt{\omega_{1}}}
\right),
\\
&& C_{2n + 1}
< 
2^p
\frac{1}{\sqrt{\omega_{2n + 2}}}
\left(
\frac{1}{\sqrt{\omega_{2n + 2}}} +
\frac{1}{\sqrt{\omega_{2n}}} +
\cdots +
\frac{1}{\sqrt{\omega_{4}}} 
+
\frac{1}{\sqrt{\omega_{2}}}
\right).
\end{eqnarray*}
By the inequality
\begin{eqnarray*}
\sum_{n = 1}^\infty C_n < 
\sum_{n = 1}^\infty
2^p
\frac{1}{\sqrt{\omega_{n + 1}}}
\sum_{l = 1}^\infty
\frac{1}{\sqrt{\omega_{l}}}
<
2^p \left(
\sum_{n = 1}^\infty n^{- p / 2} 
\right)^2,
\end{eqnarray*}
we conclude that $\{\omega_n = n^p\}_{n = 1}^\infty$ have property (SC).
\end{proof}

By the same discussion in the proof of Theorem\ref{qGauss}, we obtain the following:

\begin{theorem}\label{TheoremHSD}
Let $p$ be a positive number.
The probability measure corresponding to the Jacobi sequence
$\{n^p\}$ is unique if and only if $p \leqq 2$.
\end{theorem}

Applying the classical result on orthogonal polynomials (due to Meixner \cite{Meixner}), 
it is easy to show that the sequence is corresponding 
to the probability measure $\displaystyle{\frac{1}{e^{\pi x / 2} + e^{- \pi x / 2}} d x}$, which is now called 
``hyperbolic secant distribution.''
The equation $(5, 4)$ of \cite{Meixner}
gives a sequence of polynomials 
$\{P_n(x)\}_{n = -1, 0, \cdots}$ satisfying 
$x P_n = P_{n + 1} + n^2 P_{n -1}$
in the case that $l_1 = 0$, $\lambda = 0$, $k_2 = -1$ 
and $\kappa = -1$.
The last paragraph of \cite{Meixner} proves that
the three-term recurrence relation is realized by 
the probability measure $d \psi$.
The measure is a scaler multiple of
\[\Gamma \left( \frac{i x + 1}{2} \right) 
\Gamma \left( \frac{- i x + 1}{2} \right) dx
=
\frac{\pi}{\mathrm{sin}((ix + 1)\pi / 2)} dx
=
\frac{\pi}{\mathrm{cos}(i x \pi / 2)} dx.
\]
Corollary \ref{TheoremHSD} means that for the case of $\omega_n = n^p$ the hyperbolic secant distribution
is ``the last probability measure'' which is uniquely determined by 
the moment sequence.

\section{Determinate moment problem and property (SC)}
\label{sectiondeterminate}

\begin{theorem}\label{equivalence}
Let $\{M_n \}_{n =0}^\infty$ be a sequence of moments.
Suppose that $M_n = 0$ for every odd number $n$.
Let $\{\omega_n \}_{n = 1}^\infty$ 
denote the corresponding Jacobi sequence.
Suppose that the Jacobi sequence
$\{\omega_n\}_{n = 1}^\infty$ is of infinite type.
Then the following conditions are equivalent:
\begin{enumerate}
\item\label{ConditionSC}
The sequence $\{\omega_n\}_{n = 1}^\infty$ has property (SC);
\item\label{ConditionNotUnique}
A probability measure $\mu$ satisfying
$M_n = \int_{x \in \mathbb{R}} x^n d \mu$
is not unique.
\end{enumerate}
\end{theorem}
We have already shown that if condition (\ref{ConditionSC}) holds, 
then there exist two probability measures $\mu_\mathrm{even}$ and $\mu_\mathrm{odd}$ which are distinguished by
the Stieltjes transforms at $i$.
We prove that condition (\ref{ConditionNotUnique}) implies 
(\ref{ConditionSC}).

\begin{proof}
Suppose that a probability measure $\mu$ satisfying
$M_n = \int_{x \in \mathbb{R}} x^n d \mu$
is not unique.
By Theorem 2 of Simon \cite{Simon}, the symmetric operator $X$ is not essentially self-adjoint.
Then the kernel $\mathrm{ker} (i - X^*)$ or 
$\mathrm{ker} (- i - X^*)$ contains a non-zero vector $\xi$.
Replacing $i$ with $-i$, if necessary, we may assume that
$\mathrm{ker} (i - X^*)$ contains a non-zero vector 
$Z = (\,^n Z)_{n = 0}^\infty$.
Then for every $k$, we have
$\langle Z, (-i -X) \delta_k \rangle = 0$
and hence
\[i (\,^{0} Z) - \sqrt{\omega_1} (\,^1 Z) = 0, \quad
- \sqrt{\omega_{k}} (\,^{k - 1} Z) + i (\,^{k} Z) 
- \sqrt{\omega_{k + 1}} (\,^{k + 1} Z) = 0
\]
By induction, the $n$-th entry of the vector can be expressed by 
the sequence $B_n$ as
\begin{eqnarray*}
\,^n Z = 
\,^0 Z \frac{B_{n -1} i^n}
{\sqrt{\omega_1 \cdots \omega_n}}, \quad n \geqq 1.
\end{eqnarray*}
For $n \geqq 2$, by the equality $C_n = B_{n -1}/B_n$,
we have
\begin{eqnarray*}
\frac{(\,^n Z)^2}{(\,^{n-1} Z)^2} = 
\frac{B_{n -1}^2}{\omega_n B_{n -2}^2}
=
\frac{B_{n -1}^2}{(B_n - B_{n-1})B_{n-2}}
=
\frac{C_n}{(1 - C_n)C_{n -1}}.
\end{eqnarray*}
Thus we have
\begin{eqnarray*}
(\,^n Z)^2
=
(\,^0 Z)^2
\prod_{l =1}^n
\frac{C_l}{(1 - C_l)C_{l -1}}
=
(\,^0 Z)^2
\frac{C_n}{\prod_{l =1}^n (1 - C_l)}
\geqq (\,^0 Z)^2 C_n.
\end{eqnarray*}
By the inequality
$\sum_{n = 2}^\infty C_n \leqq \|Z\|^2_2 / (\,^0 Z)^2$,
we conclude that property (SC) holds.
\end{proof}

%\begin{theorem}
%The Jacobi sequence $\{n^2\}_{n = 0}^\infty$ does not have property (SC).
%\end{theorem}
%
%\begin{proof}
%We claim that $C_n = (n + 1)^{-1}$.
%Since $C_1 = 1/ (1 + 1) = 2^{-1}$, this equality holds for $n = 1$.
%Suppose that $C_n = (n + 1)^{-1}$ holds.
%Then we have 
%\[
%C_{n + 1} = \frac{1}{1 + \omega_{n+1} C_n}
%= \frac{1}{1 + (n+1)^2 / (n + 1)}
%= \frac{1}{n + 2}.
%\]
%It follows that $\sum_{n = 1}^\infty C_n = \infty$.
%\end{proof}

\begin{acknowledgment} 
The authors are grateful to Prof. M. Bo\.{z}ejko for many suggestions and comments to refine the drafts.
\end{acknowledgment}

\bibliographystyle{amsalpha}
\bibliography{indeterminate.bib}

\end{document}